\documentclass[reqno]{amsart}

\usepackage{color}
\usepackage{amsmath}
\usepackage{amssymb}
\usepackage{amsfonts}
\usepackage{amscd}
\usepackage{amsthm}
\usepackage{latexsym}
\usepackage{graphicx}

\usepackage[subnum]{cases}

\usepackage{amssymb}

\def\be{\begin{equation}}
\def\ee{\end{equation}}
\def\bse{\begin{subequations}}
\def\ese{\end{subequations}}

\newtheorem{thm}{Theorem}

\newtheorem{rem}[thm]{Remark}

\def\bse{\begin{subequations}}
\def\ese{\end{subequations}}

\title[Blow-up for semilinear filtration problems]{On the finite time blow-up for filtration problems with nonlinear reaction}

\author{K. Fellner}

 \address{Klemens Fellner \hfill\break
Institute for Mathematics and Scientific Computing, Karl-Franzens-University Graz 
Heinrichstr. 36 - A-8010 Graz, Austria}
\email{Klemens.Fellner@uni-graz.at}

 \author{E. Latos}
 \address{Evangelos Latos \hfill\break
Institute for Mathematics and Scientific Computing, Karl-Franzens-University Graz 
Heinrichstr. 36 - A-8010 Graz, Austria and\hfill\break
University of Mannheim, D-68161 Mannheim, Germany}
 \email{evangelos.latos@math.uni-mannheim.de}

\author{G. Pisante}

\address{Department of Mathematics and Physics, Seconda Universit\`a degli Studi di Napoli,
Viale Lincoln, 5
- 81100 Caserta,
Italy}

\begin{document}

\begin{abstract}
We present results for finite time blow-up for filtration problems with nonlinear reaction under appropriate assumptions on the nonlinearities and the initial data. 
In particular, we prove first finite time blow up of solutions subject to sufficiently large initial data provided that the reaction term "overpowers" the nonlinear diffusion in a certain sense.
Secondly, under related assumptions on the nonlinearities, we show that initial data above positive stationary state solutions will 
always lead to finite time blow up.
\end{abstract}



\maketitle


\section{Introduction}
In this paper, we shall study the blow-up of solutions $u=u(x,t)$ to the following filtration 
problem
\bse\label{filtr}
\begin{align}
&\ u_t=\Delta
K(u)+\,\lambda f(u),&&x\in \Omega, \quad t>0,\label{eq122}\\
&\mathcal{B}(K(u))\equiv\hat
n\cdot \nabla K(u)+\beta(x)K(u)=0,
&&x\in \partial\Omega,\quad t>0,\label{eq123}\\
&u(x,0)=u_0(x)\geq 0,&& x\in \Omega, 
\end{align}
\ese
where $\Omega$ is a bounded domain of $\mathbb{R}^N$ with sufficient smooth boundary $\partial\Omega$ and $\hat{n}$ denotes the outer unit normal vector. 

In problem \eqref{filtr}, the non-linear functions $f$ and $K$, are supposed to be in $C^3(\mathbb{R})$ and to satisfy the following 
positivity, growth, monotonicity and convexity assumptions:
\be\label{314}
K(0) \ge 0,\ K(s)>0, \quad \mbox{for}\ s\in\mathbb{R}^{+}, \qquad \mbox{and}\qquad  K'(s),\
K''(s)>0, \quad \mbox{for}\ s\in\mathbb{R}^{+}_{0},
\ee
\be \label{fun13333}f(s)>0,\ f'(s)>0,\ f''(s)>0, \quad
\mbox{for}\;\;s\in\mathbb{R}_{0}^{+}, \ee
\be\label{fun2444}\int_{0}^{\infty}\frac{K'(s)}{f(s)}ds<\infty,\quad \mbox{which implies}\;\; \int_{0}^{\infty}\frac{ds}{f(s)}<\infty.
\ee
Moreover, we assume the constant $\lambda>0$ to be positive and the coefficient  $0\leq\beta(x)\leq \infty$ to be in
$C^{1+\alpha}(\partial\Omega)$ for $\alpha>0$ wherever it is bounded. Note that 
$\beta\equiv0$, $\beta\equiv\infty$ and $0<\beta<\infty$ specify homogeneous
Neumann, Dirichlet and Robin boundary conditions, respectively. Moreover, this type of conditions are a consequence of Fourier's law for diffusion when considering either conservation of mass, or conservation of energy (see for example \cite{va}).

We remark that by imposing non-negative initial data $u_0(x)\geq 0$, the assumed positivity $f(u)>0$ implies the positivity of solutions to \eqref{filtr}, i.e. $u>0$ in $\Omega$ for $t>0$. For the existence of a unique classical local solution $u\in C^{2,1}(\Omega_T)$  to problem \eqref{filtr}--\eqref{fun13333} we refer to \cite{vldt2,vldt,pqps,sa1,soup,va} and the references therein.

Under the Osgood type condition \eqref{fun2444}, which is necessary for the blow-up of solutions, the filtration problem \eqref{filtr} will exhibit a blow-up behaviour if the forcing term is sufficiently "strong" and the initial data are sufficiently "large". This can be illustrated by the following example, which applies Kaplan's method, i.e. it investigates the evolution 
of the Fourier coefficient 
\begin{equation}\label{Kaplan}
B(t):= \int_\Omega u(x,t) \,\phi(x)\, dx,\qquad
\dot{B}(t)=
\int_{\Omega} \Delta K(u)\,\phi\,dx+\lambda \int_{\Omega}
f \phi\,dx,
\end{equation}
where the assumptions on $\beta$ ensure that $\phi(x)$ can be chosen as the positive (in $\Omega$) and $L^1$-normalised (i.e. $\|\phi\|_1=1$) eigenfunction corresponding to the first eigenvalue $\mu$ of the following auxiliary elliptic problem with Robin boundary conditions:
\be
\label{robin}
\begin{cases}
\Delta \phi+\mu\phi=0, &  \qquad x\in
\Omega,  \\
 \hat{n}\cdot\nabla\phi+\beta(x)\phi=0, &  \qquad x\ \mbox{on}\
\partial\Omega.
\end{cases}
\ee
After integration by parts in \eqref{Kaplan} and using the eigenvalue problem (\ref{robin}) we obtain
$
\dot{B}(t)=
-\mu\int_{\Omega}K\,\phi\,dx+\lambda \int_{\Omega}
f \phi\,dx.
$
Hence, a comparison condition between $f$ and $K$ of the following kind
\be\label{DC1}
\forall t>0, \qquad \int_{\Omega}\left(f(u(x,t))-K(u(x,t))\right)\phi(x)\,dx \ge 0
\ee
yields, for $\lambda>\mu$ and by applying Jensen's inequality,
$$
\dot{B}(t)  \geq (\lambda-\mu)\int_{\Omega}f(u(x,t))\,\phi(x)\,dx
\geq (\lambda-\mu)f(B).
$$

Thus, the Osgood type condition \eqref{fun2444} implies the finite-time blow-up of $B(t)$ (cf. \cite[Section 4.1]{vldt2}). Moreover, as a by-product of Kaplan's method, the 
first eigenvalue $\mu$ provides a lower bound $\mu<\lambda$, above which blow-up occurs. 

These phenomena are connected with the existence of solutions to the steady-state problem corresponding to (\ref{filtr}),  
\be\label{ss}
\begin{cases}
\Delta(K(w))+\lambda f(w) = 0,  & \qquad x \in \Omega,\\
\mathcal{B}(K(w))=0, & \qquad x \in \partial \Omega.
\end{cases}
\ee
It has been shown in \cite{vldt2}, in the closed spectrum case scenario, that problem \eqref{ss} exhibits a critical (i.e maximal) 
value of the parameter $\lambda$, say $\lambda^*$, such that (any kind of) solution 
to \eqref{ss} does not exist for $\lambda>\lambda^*$, while there exist bounded 
solutions to \eqref{ss} for all $0<\lambda\le\lambda^*$. In fact, there exist at least two solutions to \eqref{ss} for $\lambda$ close to $\lambda^*$.
The case of steady-state solutions to the critical parameter $\lambda^*$ is more intricate, see \cite{vldt2,vldt}. Since $\mu \geq \lambda^*$ (see \cite{cr,kc} or \cite{vldt2} for a proof under the additional assumption \eqref{DC1}), Kaplan's method is in general not sharp enough to treat the full supercritical range $\lambda > \lambda^*$. Nevertheless, in {\cite[Section 4.2]{vldt2}} it has been shown under a mild extra condition on $f$ and $K$ that for all $\lambda > \lambda^*$ blow-up of solutions of \eqref{filtr} subject to any initial data $u_0\geq 0$ occurs (see also \cite{bel,bcmr,l}).

The aim of this work is to complement the analysis of blow-up of solutions for the above filtration problem \eqref{filtr}.
In a first result (see Theorem \ref{th6} below), we shall assume generalised comparison conditions and characterise sufficiently large initial data such that the solution to \eqref{filtr} blows up in finite time even in the subcritical region $\lambda<\lambda^*$.
Moreover, in Theorem \ref{blowSS}, we shall prove that solutions of \eqref{filtr} subject to initial data above the positive solution of the steady state problem (which exists since $\lambda<\lambda^*$) blows up in finite time.

\section{Blow-up analysis       }
In the subcritical case $\lambda < \lambda^*$, the large-time behaviour of solutions to the semi-linear filtration problem \eqref{filtr}--\eqref{fun13333} depends strongly on the initial data $u_0$. In particular, as one would expect, blow-up occurs for large enough initial data. A first result in this direction can again be proven by Kaplan's method. More precisely, we can state the following:

\begin{thm}[Subcritical Blow-Up for Large Initial Data]\label{th6}\hfill\\
Let  the assumptions \eqref{314}, \eqref{fun13333} and \eqref{fun2444} hold. Let $\mu$ be the first eigenvalue of the problem \eqref{robin} and $\phi$ be the corresponding positive normalised eigenfunction.  Let  $u(x,t)$ denote the solution of \eqref{filtr}. 
Assume either that there exists a \emph{concave} function $\psi: \mathbb{R}_+\mapsto \mathbb{R}_+$ 
 satisfying $\gamma:=\limsup_{s\to \infty} \frac{\psi(s)}{s} < \frac{\lambda}{\mu}$ and
such that  
\be\label{lidc}
\int_{\Omega}\left[ \psi(f(u(x,t)))-K(u(x,t)\right]\phi(x)\,dx \ge 0, \qquad \forall t>0,
\ee
or that there exists a \emph{convex} function $\varphi: \mathbb{R}_+\mapsto \mathbb{R}_+$ with $\varphi(0)=0$ and  $\kappa :=\liminf_{r\to \infty}  \frac{\varphi(r)}{r} > \frac{\mu}{\lambda}$
such that
\be\label{lidd}
\int_{\Omega}\left[ f(u(x,t))-\varphi(K(u(x,t))\right]\phi(x)\,dx \ge 0, \qquad \forall t>0,\quad \text{and} \quad \int_{1}^{\infty}\frac{ds}{K(s)}<\infty.
\ee
Then, the solution $u(x,t)$ of \eqref{filtr}--\eqref{fun13333} blows up in finite time, provided that the initial data $u_0 \ge 0$ are sufficiently large.
\end{thm}
\begin{rem}
The proof of Theorem \ref{th6} entails a sufficient condition on initial data $u_0$
to yield finite time blow-up of solutions to \eqref{filtr} under the above assumptions. More precisely, blow-up occurs for initial data satisfying
$$
f(B(0))=f\left(\int_{\Omega} u_0\,\phi\, dx\right)  > s_0, \qquad\text{where}\quad
s_0 = \inf_{s\in\mathrm{R}_+}\{\lambda\, {s} > \mu\,\psi(s)\}\ge0,
$$
or 
$$
K(B(0))=K\left(\int_{\Omega} u_0\,\phi\, dx\right)  > r_0, \qquad\text{where}\quad
r_0 = \inf_{r\in\mathrm{R}_+}\{\lambda\, \varphi(r)> \mu\, r  \}\ge 0.
$$
In particular, if $\psi$ is such that $s_0=0$ holds, then finite time blow-up occurs unconditionally for all initial data $u_0\ge0$, since assumption \eqref{fun13333} implies always $f(0)>0=s_0$.
\end{rem}

\begin{proof} [Proof of Theorem \ref{th6}]
At first, by multiplication of \eqref{eq122} with the eigenfunction $\phi$ and by integration over $\Omega$, we obtain the following evolution law for the 
Fourier coefficient $B(t)$,
$$
\dot{B}(t)=\frac{d}{d t}\int_{\Omega}u\,\phi\, dx=
\int_{\Omega} \Delta K(u)\,\phi\,dx+\lambda \int_{\Omega}
f(u)\, \phi\,dx.
$$ 
Assuming, in order to get a contradiction, that $u$ is a global-in-time solution, we apply Green's identity and remark that the boundary terms cancel since the eigenproblem \eqref{robin} is chosen to satisfy the same boundary condition as the filtration problem \eqref{eq123}. Thus, we obtain
\be\label{aa}
\dot{B}(t)=-\mu\int_{\Omega} K(u)\, \phi \,dx +\lambda
\int_{\Omega} f(u)\, \phi\,dx.
\ee 

We shall start by considering assumption \eqref{lidc}. By inserting assumption \eqref{lidc} into \eqref{aa} and applying Jensens's inequality with respect to the normalised positive measure $\phi(x)\,dx$, we get the estimate 
\begin{equation}
\dot{B}(t)
\ge  \lambda \int_\Omega f(u)\, \phi \,dx - \mu \int_\Omega  \psi(f(u))\, \phi \,dx 
\ge \lambda \int_\Omega f(u)\,\phi\,dx - \mu\, \psi\left(\int_\Omega f(u)\,\phi\,dx\right).
\label{aaa}
\end{equation}
Next, we define $s(t) := \int_\Omega f(u)\,\phi\,dx$. Due to the convexity of $f$, we observe that, again by Jensen's inequality, 
\begin{equation}\label{sB}
s(t)=\int_\Omega f(u)\,\phi\,dx\ge f\left(\int_\Omega u\,\phi\,dx\right) = f(B(t)).
\end{equation}

Then, we set for any $0<\varepsilon < \lambda-\mu \gamma$ 
\begin{equation}\label{sel}
s_{\varepsilon} := 
\inf_{s \in\mathbb{R}_+}\left\{\frac{\lambda -\varepsilon}{\mu}  > \frac{\psi(s)}{s} \right\}.
\end{equation}
Note that assumptions on $\psi$  
imply that $0\le s_{\varepsilon}<+\infty$ and we have $\lambda s-\mu \,\psi(s) > \varepsilon s$ for all $s>s_{\varepsilon}$. Therefore, we can continue to estimate \eqref{aaa} by using \eqref{sB} and \eqref{sel} to obtain
\begin{equation}\label{***}
\dot{B}(t) \ge \varepsilon s(t) \ge \varepsilon f(B(t))>0,\qquad\qquad \text{for all}\quad 
s(t)\ge s_{\varepsilon}.
\end{equation}
The differential inequality \eqref{***} implies the finite time blow-up of the Fourier coefficient $B(t)$ 
provided that the initial value $B(0)$ is chosen large enough such that $s(0)\ge f(B(0))\ge  s_{\varepsilon}$. Moreover, the blow-up occurs at latest at time 
$
T^*_{\varepsilon} \le \frac{1}{\varepsilon} \int_{B(0)}^{\infty}\frac{ds}{f(s)} <+\infty$, provided that $ B(0)\ge f^{-1}(s_{\varepsilon})
$.
Thus, minimisation in $\varepsilon$ yields blow-up the latest at time
\begin{equation*}
T^* \le \inf_{0<\varepsilon<\lambda-\mu \gamma}
\left\{\frac{1}{\varepsilon} \int_{f^{-1}(s_{\varepsilon})}^{\infty}\frac{ds}{f(s)} \right\}<+\infty.
\end{equation*}

We shall now consider assumption \eqref{lidd} and proceed in an analogous way to above. Using \eqref{lidd} in \eqref{aa} and applying Jensen's inequality for the convex function $\varphi$ yields
\begin{equation}
\dot{B}(t)
\ge  \lambda \int_\Omega \varphi(K(u))\, \phi \,dx - \mu \int_\Omega  K(u)\, \phi \,dx 
\ge \lambda\, \varphi\left(\int_\Omega K(u)\,\phi\,dx\right) - \mu \int_\Omega K(u)\,\phi\,dx.
\label{bbb}
\end{equation}
Similar to above, we define $r(t) := \int_\Omega K(u)\,\phi\,dx$ and observe as in \eqref{sB} that the convexity of $K$ and Jensen's inequality with respect to the normalised positive measure $\phi$ imply 
$r(t)\ge K(B(t))$.

Then, we denote, for any $0<\varepsilon < \kappa \lambda-\mu$, $r_{\varepsilon} := \inf_{r \in\mathbb{R}_+}\left\{\lambda \varphi(r)-\mu r > \varepsilon r \right\}$.
Note that the assumptions on $\varphi$ imply that $0\leq r_{\varepsilon}<+\infty$ and we have that $\lambda \varphi(r)-\mu r > \varepsilon r$ for all $r>r_{\varepsilon}$. Therefore, we can continue to estimate \eqref{bbb} as
\begin{equation*}
\dot{B}(t) \ge \lambda\, \varphi(r(t))-\mu \,r(t) > \varepsilon\, r(t) \ge \varepsilon \,K(B(t))>0,\qquad\qquad \text{for all}\quad 
r(t)\ge r_{\varepsilon}.
\end{equation*}
Thus, analog to above, we obtain finite time blow-up latest at time
\begin{equation*}
T^* \le \inf_{0<\varepsilon< \kappa \lambda-\mu}
\left\{\frac{1}{\varepsilon} \int_{K^{-1}(r_{\varepsilon})}^{\infty}\frac{ds}{K(s)} \right\}<+\infty.
\end{equation*}
This finishes the proof.
\end{proof}

\section{Blow-up for initial data above positive steady state}
\label{sec-steady}

In this section, we will follow ideas of \cite{l} and \cite{soup} to prove that solutions of the filtration problem \eqref{filtr} subject to initial data $u_0$ larger than a positive solution of the associated stationary problem will blow-up in finite time provided suitable conditions on the involved nonlinearities hold, in particular that the function $g(s):=f(K^{-1}(s))$ is convex. Thus, under this assumption, we are able to sharpen the result of Theorem \ref{th6}. We remark also that solutions being initially below a positive steady state remain naturally bounded by the steady state solution due to a comparison principle (see \cite{vldt}) and thus exist globally in time.

In the following, we say that $w =w(x)> 0$ is a classical solution of \eqref{ss},
if $z=z(x)=K(w(x))$ is a classical solution 
of
\be\label{ss2}
\begin{cases}
\Delta z+\lambda g(z) = 0, & x \in \Omega, \\
\mathcal{B}(z)=0, & x \in \partial \Omega,
\end{cases}
\ee
where $g(z)=f(K^{-1}(z))=f(w)$ with $z=K(w)$. 
The assumptions \eqref{314}, in particular the monotonicity property of $K$ implies that both the above
steady-state problems are equivalent with respect to the existence and the number of solutions. For a positive steady-state solution $w>0$ of \eqref{ss} (see \cite[Section 3.2]{vldt}), the eigenvalue problem associated to the 
linearisation of the filtration problem \eqref{filtr} around $w$ (i.e. we expand \eqref{filtr} according to $u(t,x)=w(x)+\Phi(t,x)$
and decompose $\Phi(t,x)=e^{-\mu t} \phi(x)$),  reads as 
\be\label{lin}
 \begin{cases}
        -\Delta [K'(w)\phi] = \lambda f'(w)\phi + \mu \phi, & \textrm{ in } \Omega, \\
        \mathcal{B}(K'(w)\phi) =0,  & \textrm{ on } \partial \Omega .
 \end{cases}
\ee
We remark that the linearisation eigenvalue problem \eqref{lin} has a solution for each  $\lambda\in(0,\lambda^*]$ (cfr. \cite[Theorem 8]{vldt}). 
Moreover, it has been shown that the first eigenvalue $\mu(\lambda)$ of \eqref{lin} is negative if the function $g(s)=f(K^{-1}(s))$ is convex, for instance.

\begin{thm}[Blow-up Above Positive Steady States]\label{blowSS}
Let $\Omega$ be a bounded domain of class $C^{2,\alpha}$ and $w>0$ be a classical positive solution of the stationary-state problem \eqref{ss}. Let $u\in C(\overline{\Omega} \times [0,T))$ be a positive solution of \eqref{filtr} and $T>0$ be the maximal time of existence. Assume that the function $g(s):=f(K^{-1}(s))$ is convex. Assume initial data $u_{0}\geq w$ in $\Omega$ with $u_0\not=w$. Then, the solution blows-up in finite time.
\end{thm}

 \begin{rem}
For Neumann and Robin boundary conditions, the proof of Theorem \ref{blowSS} can be generalised to allow for $K'(0)=0$ (since then we have $m:=\inf_{x\in \Omega} w(x) > 0$), which allows one to include the porous medium problem.
\end{rem}
\begin{proof}[Proof of Theorem \ref{blowSS}]

Let $\phi$ be the first eigenfunction of problem \eqref{lin} corresponding to a negative eigenvalue $\mu(\lambda)$. We consider the function 
$
\theta (x,t) := u(x,t)-w(x) \geq 0,
$
which being initially non-negative remains non-negative thanks to the comparison principle proven in \cite{vldt}. Thus, the monotonicity assumption on $K$ also ensures $ K(u)-K(w) \geq 0$.  Moreover, by using the equations \eqref{eq122} and \eqref{ss}, we have 
$
\theta_{t}= \Delta(K(u)-K(w))+\lambda (f(u)-f(w)). 
$

Next, we test $\theta_{t}$ with $K'(w)\,\phi$, integrate by parts and use that $w$ is a solution of \eqref{lin} to calculate
\begin{align}
\frac{d}{dt}\int_{\Omega} & \theta K'(w)\,\phi\, dx  =  \int_{\Omega} (K(u)-K(w))\,\Delta [K'(w)\,\phi] \, dx \nonumber \\
& + \lambda \int_{\Omega}  (f(u)-f(w))\,K'(w)\, \phi\, dx \nonumber \\
 = &  \int_{\Omega} (K(u)-K(w))(-\lambda f'(w)\,\phi - \mu \phi) \, dx + \lambda  \int_{\Omega}  (f(u)-f(w))\,K'(w)\, \phi\, dx \nonumber \\
 = &\  \lambda  \int_{\Omega}  \big[ K'(w) (f(u)-f(w)) -f'(w) (K(u)-K(w)) \big]\, \phi \, dx \nonumber \\
 & - \mu  \int_{\Omega}  (K(u)-K(w))\, \phi\, dx. \label{AA}
\end{align}
Now, we introduce the function $h(s):=f(s)-f(0)-sf'(0)$. It is easily verified that $h$ satisfies 
\be \label{conh} h(s)>0, \;\;s\in\mathbb{R^+}, \;\;h'(s)>0,
\;\;h(0)=h'(0)=0,\;\; h''(s)>0,\;\;s\in\mathbb{R},\ee
and its minimum is $(0,\,h(0))=(0,0)$. 

In the following, we apply a dichotomy argument (see also \cite[pag. 1355]{l} and \cite[proof of Theorem 9]{vldt} for its generalisation to $K(u)$ satisfying assumptions \eqref{314} in the case of supercritical $\lambda>\lambda^*$) in order to prove that there exists a constant 
$\Lambda>0$ such that, for any $s\geq 0$ and independently on $x\in \Omega$, the following estimate holds:
\be\label{major}
F_w(s):= K'(w) (f(w+s)-f(w)) -f'(w) (K(w+s)-K(w)) \geq \Lambda\, h(s).
\ee
Clearly it is enough to study the case $s>0$. To this aim, we start by defining the constants
\[
m:=\inf_{x\in \Omega} w(x) \geq 0,\;\;\;\; M:= \sup_{x\in \Omega} w(x)< + \infty,\;\;\;\; c_1:= \frac{f'(M)}{K'(m)},\;\;\;\;c_2:=\frac{f(M)}{K'(m)}.
\]
Using the growth condition \eqref{fun2444} we can ensure the existence of $S=S(f,K,\lambda)$, such that
\be \label{line}
\frac{f(s)}{K'(s)} > 2(c_1 \,s +c_2),\;\;\; \forall \, s>S,
\ee
it is sufficient indeed to chose $S=\sup\left\{ t\;:\; \frac{f(t)}{K'(t)} - 2(c_1 \,t +c_2)=0\right\}$. The dichotomy consist in dealing with the case $s>S$ and $s\leq S$ separately. First, for $s>S$, we can estimate
$
 f(w+s)  > 2[ c_1 (w+s)  + c_2] K'(w+s)  
         \geq 2 [s \,c_1 \, K'(w+s) + f(M)].
$
Using the previous inequality, since
\[
f'(w) (K(w+s)-K(w)) = K'(\xi)\,s f'(w) \leq K'(w+s)\,s\,f'(M),
\] 
by choosing $\Lambda_1$ such that $0<\Lambda_1< K'(m)/2$ and with the help of the above relation we can deduce, again for $s>S$,
\be\label{destra}
\begin{split}
F_w(s)  \geq & K'(w) (f(w+s)-f(w) - c_1 \, s \, K'(w+s)) \\
                \geq & K'(w) (f(w+s)-f(M) - c_1 \, s \, K'(w+s)) \\
                \geq & \frac{1}{2} K'(w)\, f(w+s) 
                 \geq  \frac{1}{2} K'(w)\, f(s) \\
                \geq &  \Lambda_1 (f(s)-f(0)-sf'(0)) = \Lambda_1 h(s) > 0.
\end{split}
\ee
Now we study the other case, $s\leq S$. We start by using the Taylor expansion up to the second order to rewrite $F_w(s)$ as
\[
F_w(s)= \big( f''(w+\eta) K'(w)-K''(w+\eta)f'(w) \big) \frac{s^2}{2},
\]
for some $0<\eta< s$. The hypothesis on $g(s)$  gives 
\be\label{neasxe1}\frac{f''(\theta)}{K''(\theta)}>\frac{f'(\theta)}{K'(\theta)}>\frac{f'(w)}{K'(w)},\;\;\;\theta=w+\eta>0,\;\;\mbox{with}\;\;0\leq\eta\leq S.\ee
Hence, from (\ref{neasxe1}) we have:
\[
\begin{split}
F_w(s) & = \big[f''(w+\eta)K'(w)-K''(w+\eta)f'(w)\big]\frac{s^2}{2} \\
         &\geq  \min_\eta\big [f''(w+\eta)K'(w)-K''(w+\eta)f'(w)\big]\frac{s^2}{2} \\
        & \geq\big[f''(w+\tilde\eta)K'(w)-K''(w+\tilde\eta)f'(w)\big]\frac{s^2}{2} \\
         &\geq \min_{x\in\overline{\Omega}}\big[f''(w+\tilde\eta)K'(w)-K''(w+\tilde\eta)f'(w)\big]\frac{s^2}{2}\\
        &\geq\big[f''(w(\tilde{x})+\tilde\eta)K'(w(\tilde{x}))-K''(w(\tilde{x})+\tilde\eta)f'(w(\tilde{x}))\big]\frac{s^2}{2},
\end{split}
\]
where we have denoted by $\tilde\eta\in [0, S]$ and $\tilde x\in \overline{\Omega}$ the values and points, respectively, which attain the minimum and the maximum, respectively, in the previous inequalities. 
Moreover, thanks  to (\ref{neasxe1}) we can easily verify that 
\[
\Lambda_2:=\frac{1}{2}\big[f''(w(\tilde{x})+\tilde\eta)K'(w(\tilde{x}))-K''(w(\tilde{x})+\tilde\eta)f'(w(\tilde{x}))\big] >0.
\]
So, since $f''(0)>0$, there exists a $\Lambda_3>0$, small enough such that:
\be\label{1}F_w(s)\geq \Lambda_2 s^2\geq
\Lambda_3[f(s)-f(0)-sf'(0)]= \Lambda_3 h(s) > 0,\quad  \text{ if } s\leq S.
\ee
Choosing $\Lambda = \min\{\Lambda_1 , \Lambda_3\}$ and combining \eqref{1} and \eqref{destra} we get \eqref{major}.
\medskip

Finally, in order to prove the statement of the theorem, we denote
$
A(t):=\int_{\Omega} \theta K'(w)\,\phi\, dx,
$ 
and continue to estimate \eqref{AA}, 
\begin{align}
A'(t) &  \geq \lambda  \int_{\Omega}  \big[ K'(w) (f(u)-f(w)) -f'(w) (K(u)-K(w)) \big] \,\phi \, dx  \nonumber \\
&\quad - \mu \int_\Omega (K(u)-K(w))\, \phi\, dx \nonumber \\
 & \geq  \frac{\lambda \Lambda}{\|K'(w)\|_{\infty}} \int_\Omega h(\theta)K'(w)\, \phi \, dx - \mu \int_\Omega (K(u)-K(w))\, \phi\, dx\nonumber\\ 
 &\geq\frac{\lambda \Lambda}{\|K'(w)\|_{\infty}} h\left(\int_{\Omega} \theta K'(w)\,\phi\, dx\right) =  \lambda \Lambda \,h(A(t)).\label{AAA}
 \end{align}
Here, in the last inequality, we have used Jensen's  inequality for the convex function $h$ and the normalized
measure $K'(w)\,\phi\,dx$, i.e. $\int_{\Omega} K'(w)\,\phi\, dx=1$ as well as $\mu<0$ and 
$\int_\Omega (K(u)-K(w)) \phi\, dx \geq 0$ for all $t\ge0$. 

The inequality \eqref{AAA}, together with \eqref{major}, implies the blow-up of $A(t)$ and since $A(t)\leq||\theta(\cdot,t)||$,   blow-up of $\theta$ at $t_\theta^*$ and hence of $u$ at $t^*<\infty$ where $t^*\leq t^*_\theta$, since $u=w+\theta>\theta$  in $\Omega$ and $w$ is bounded.
\end{proof}












\end{document}